\documentclass[12pt]{amsart}
\usepackage{amssymb}
\usepackage{dsfont}
\usepackage{color}
\usepackage[utf8]{inputenc}
\usepackage[T1]{fontenc}

\usepackage{hyperref}
\usepackage{amsfonts}
\usepackage{eurosym}
\usepackage{dsfont}
\usepackage{color}
\usepackage{graphicx}
\usepackage{mathtools}
\usepackage{tikz}

\newtheorem{theorem}{Theorem}
\newtheorem{lemma}[theorem]{Lemma}

\newtheorem{corollary}[theorem]{Corollary}

\newcommand{\R}{\mathbb R}
\newcommand{\C}{\mathfrak c}

\newcommand{\w}{\omega}
\newcommand{\Li}{\mathcal{L}}

\author{Szymon Głąb}
\address{Institute of Mathematics, Lodz University of Technology, Wólczańska 215, 93-005 Łódź, Poland}
\email {szymon.glab@p.lodz.pl}
\author{Mateusz Lichman}
\address{Institute of Mathematics, Lodz University of Technology, Wólczańska 215, 93-005 Łódź, Poland}
\email {mateusz.lichman@wp.pl}
\author{Michał Pawlikowski}
\address{Institute of Mathematics, Lodz University of Technology, Wólczańska 215, 93-005 Łódź, Poland}
\email {michal-pawlikowski4@wp.pl}
\title[On strong algebrability of families of non-measurable functions]{On strong algebrability of families of non-measurable functions of two variables}
\date{}

\begin{document}

\begin{abstract}
Recently Tomasz Natkaniec in \cite{Nat} studied the lineability problem for several classes of non-measurable functions in two variables. In this note we improve his results in the direction of algebrability. In particular, we show that most of the classes considered by Natkaniec contain free algebras with $2^\C$ many generators. 
\end{abstract}
\maketitle

\section{Introduction}
The last 20 years have seen a huge development in the study of the existence of large and rich algebraic structures within the subsets of linear spaces, function algebras and their Cartesian products. The topic already has its own place in the Mathematical Subject Classification -- 46B87, and both a monograph (see \cite{LineabilityBook}) and a review article (see \cite{LineabilityArticle}) are devoted to it. The customary name for problems in this area is lineability or algebrability problems. These problems occur in many areas of mathematics.

Recently Tomasz Natkaniec in \cite{Nat} considered the lineability problem for several classes of non-measurable functions of two variables. Most of his results are optimal in the sense that given families are $2^\C$-lineable in the algebra of all real functions of two variables which is itself of cardinality $2^\C$. 

Improving all the results of \cite{Nat} in the direction of algebrability is the main goal of this paper:
\begin{itemize}

    \item In \cite[Theorem 3 and Theorem 9]{Nat} it is proved that the family of all sup-measurable functions that are non-measurable is $2^\C$-lineable; first it is proved under CH, then under $\operatorname{non}(\mathcal{N})=\mathfrak{c}$. Both of these imply condition (A) (see Section \ref{SectionConitionA}), which in turn implies that the family is strongly $2^\C$-algebrable, see Theorem \ref{SupMeasurableNonMeasurableNewAssumption}.
    \item In \cite[Theorem 4 and Theorem 10]{Nat} it is proved that the family of all weakly sup-measurable functions which are neither measurable nor sup-measurable is $2^\C$-lineable; first it is proved under CH, then under $\operatorname{non}(\mathcal{N})=\C$. We prove in Theorem \ref{WeaklySupmeasurableNonmeasurableNonSupmeasurable} that (A) implies that the family is strongly $2^\C$-algebrable.
    \item In \cite[Theorem 12]{Nat} it is proved that the family of all non-measurable separately measurable functions (see Section \ref{SectionSeparatelyMeasurable}) is $2^\C$-lineable. We prove that this family is strongly $2^\C$-algebrable, see Theorem \ref{Non_measurable_sep_measurable}.
     \item In \cite[Theorem 13]{Nat} it is proved that the family of all non-measurable functions $F\colon\R^2\to\R$ whose all vertical and horizontal sections are Darboux Baire one is $\C$-lineable. We prove that this family is strongly $\C$-algebrable, see Theorem \ref{TheoremDarbouxSections}. Our proof is relatively simple compared to Natkaniec's.  
    \item In \cite[Theorem 16]{Nat} it is proved that the family of all non-measurable functions having all vertical sections approximately continuous (see Section \ref{Appr_Cont_Section}) and all horizontal sections measurable is $2^\C$-lineable under the assumption that $\operatorname{cov}(\mathcal{N})=\operatorname{add}(\mathcal{N})$ (see Section 
    \ref{Cov_is_Add_section}).  In Theorem \ref{Non_Measurable_all_vert_sec_app_cont_and_all_hor_sec_meas} we improve it to strong $2^\C$-algebrability.
\end{itemize}
For completeness, we show that the family of all measurable functions that are not sup-measurable is strongly $2^\C$-algebrable, see Theorem \ref{Mateusz}. Furthermore, in Section \ref{Jones_Secion} we define a family of sup-Jones functions. We prove that this family is $2^\C$-lineable, see Theorem \ref{Jonesy}.

The paper is organised as follows. In Section \ref{Pre} we give all the ingredients. We have divided it into several subsections to help the reader navigate. In Section \ref{Res} we cook up the proofs. 

\section{Preliminaries}\label{Pre}
\subsection{Lineability and strong algebrability}
Let $\mathcal{L}$ be a vector space, $A \subseteq \mathcal{L}$ and $\kappa$ be a cardinal number. We say that $A$ is $\kappa$\textit{-lineable} if $A \cup \{0\}$ contains a $\kappa$-dimensional subspace of $\mathcal{L}$. 
If we take $\mathcal{L}$ to be a commutative algebra, $A \subseteq \mathcal{L}$, then we say that $A$ is \textit{strongly $\kappa$-algebrable} if $A \cup\{0\}$ contains a $\kappa$-generated subalgebra $B$ which is isomorphic to a free algebra.

Note that the set $X=\left\{x_{\alpha}\colon \alpha<\kappa\right\}$ is a set of free generators of some free algebra if and only if the set of all elements of the form $x_{\alpha_{1}}^{k_{1}} x_{\alpha_{2}}^{k_{2}} \cdots x_{\alpha_{n}}^{k_{n}}$,  where $k_1,k_2,\dots,k_n$ are non-negative integers non-equal to $0$ and $\alpha_1<\alpha_2<\ldots<\alpha_n<\kappa$, is linearly independent; equivalently, for any $k \geq 1$, any non-zero polynomial $P$ in $k$ variables without a constant term and any distinct $x_{\alpha_{1}}, \ldots, x_{\alpha_{k}} \in X$, we have that $P(x_{\alpha_{1}}, \ldots, x_{\alpha_{k}})\in A\setminus\{0\}$. Note that if $P(x_{\alpha_1}, \dots, x_{\alpha_k})$ is non-zero for any distinct $\alpha_1, \dots, \alpha_k$, then $\{x_{\alpha}\colon\alpha < \kappa\} \subseteq A\setminus\{0\}$ (consider $P(x)=x$) and elements of $\{x_{\alpha}\colon\alpha < \kappa\}$ are  different (consider $P(x,y)=x-y$). We will use this observation without mentioning it in every single proof of algebrability or lineability.

It turns out that $\R^\R$, or equivalently $\R^\mathfrak{c}$, contains a set of free generators with cardinality $2^\mathfrak{c}$, see \cite{BGP}.

\subsection{Sup-measurable functions}
Given a real function $f$ in one variable and a real function $F$ in two variables, we can define the \textit{Carath\'{e}odory superposition} of $F$ and $f$ as a real function $F_f$ in one variable given by $F_f(x)=F(x,f(x))$. A function $F\colon\R^2\to\R$ is said to be \textit{sup-measurable} if $F_f$ is Lebesgue measurable for every Lebesgue measurable $f\colon\R\to\R$. By \cite[Lemma 3.4]{BMRS}, it is  sufficient to check the measurability of $F_f$ only for continuous functions $f$. There are measurable functions that are not sup-measurable: consider $F\colon \R^2\to \R$, $F=\chi_{X\times \{0\}}$, where $X\subseteq \R$ is non-measurable. The problem whether sup-measurable functions are measurable is undecidable in ZFC. On the one hand
under the continuum hypothesis (CH) there is a sup-measurable function that is non-measurable, see \cite{GL} and \cite{Ka} for the first such constructions. On the other hand there is a model of ZFC in which every sup-measurable function is measurable \cite{RS}. 

A function $F\colon \mathbb{R}^{2} \rightarrow \mathbb{R}$ is \textit{weakly sup-measurable} if the superposition $F_{f}$ is measurable for any continuous and almost everywhere differentiable function $f\colon \mathbb{R} \rightarrow \mathbb{R}$.

\subsection{Separately measurable functions}\label{SectionSeparatelyMeasurable} For a function $F\colon\R^2\to \R$ and $y\in \R$ we denote by $F(\cdot, y)$ the horizontal section of $F$ at $y$, i.e. the function $x\mapsto F(x, y)$. Similarly, $F(y, \cdot)$ is the vertical section of $F$ at $y$.

We say that a function $F\colon \mathbb{R}^{2} \rightarrow \mathbb{R}$ is  separately measurable if all horizontal and vertical sections of $F$ are measurable. A separately measurable function needs not to be measurable. To see this, consider a set $A\subseteq\R^2$ which has full outer measure but its intersection with each vertical and each horizontal line is a finite set (e.g. $A=\bigcup_{\alpha<\C}A_\alpha$, where $A_\alpha, \alpha<\C$, are like in Lemma \ref{rodzinaAalfa} for $Y=\R^2$). 
$A$ has full outer measure, but every vertical section of $A$ is null. Therefore, by Fubini's Theorem, $A$ is non-measurable. Consider the characteristic function $\chi_A$ of $A$. Clearly $\chi_A$ is non-measurable. Let $x\in\R$. Then $\{y\in\R\colon(x,y)\in A\}$ has at most two elements. Since $y\mapsto \chi_A(x,y)$ takes non-zero values on a finite set, it is measurable. Similarly, $x\mapsto \chi_A(x,y)$ is measurable for every $y\in\R$. So $\chi_A$ is separately measurable. Note that if $F\colon\R^2\to\R$ has the property that $\{(x,y)\colon F(x,y)\neq 0\}\subseteq A$, then $F$ is separately measurable by the very same argument. 

The following observation, which is a slight modification of \cite[Lemma 11]{Nat}, will be a useful tool for us. 
\begin{lemma}\label{rodzinaAalfa}
Let $Y\subset \R^2$ be a measurable set with positive measure. There exists a family $\left\{A_\alpha\colon \alpha<\mathfrak{c}\right\}$ of pairwise disjoint subsets of $Y$ such that
\begin{itemize}
    \item[(1)] each $A_\alpha$ has full outer measure (in $Y$); 
    \item[(2)] all horizontal and vertical sections of $\bigcup_{\alpha<\mathfrak{c}} A_\alpha$ have at most one element.
\end{itemize}
\end{lemma}

\subsection{Darboux Baire one functions}\label{sekcjaDB1}
We say that that a function $f\colon\R\to \R$ is Darboux if it has the intermediate value property. We say that $f$ is Baire one if it is a pointwise limit of a sequence of continuous functions. The latter is equivalent to the fact that $f^{-1}[U]$ is $F_{\sigma}$ for any open $U \subseteq \R$.

The next simple lemma will be a useful tool in our investigations.

\begin{lemma}\label{lemacikBaireOne}
Let $F_1, F_2, \dots, F_n$ be a partition of $\R$ into $F_{\sigma}$ sets, $f_1, \dots, f_n$ be Baire one functions. Then $\ell=\sum_{i=1}^n f_i \chi_{F_i}$ is a Baire one function.
\end{lemma}

\begin{proof}
Let $U \subseteq \R$ be open. Then

\[
  \ell^{-1}[U]=\{x \in \R\colon\sum_{i=1}^n f_i(x)\chi_{F_i}(x) \in U\}=\bigcup_{i=1}^n \{x \in F_i\colon\sum_{i=1}^n f_i(x)\chi_{F_i}(x) \in U\}
  \]
  \[
  =\bigcup_{i=1}^n \{x\in F_i\colon f_i(x) \in U\}=\bigcup_{i=1}^n (F_i \cap f_i^{-1}[U]) \]

is an $F_{\sigma}$ set, since the family of all $F_{\sigma}$ sets is closed under finite unions and intersections.
\end{proof}

\subsection{Approximately continuous functions}\label{Appr_Cont_Section}
A function $f\colon \mathbb{R} \rightarrow \mathbb{R}$ is called approximately continuous if it is continuous in the density topology, i.e. for any open set $U \subseteq \mathbb{R}$ the set $f^{-1}[U]$ is measurable and has density one at each of its points. It turns out that every approximately continuous function is Darboux of the first Baire class. If $N$ is a null set, then by \cite[Theorem 6.5]{Bruckner} there exists an approximately continuous function $g\colon\R\to[0,1]$ such that $g^{-1}(0)$ is a null cover of $N$. Then every $h_n=n\min\{g,1/n\}\colon\R\to[0,1]$ is approximately continuous as a composition of a continuous function $x\mapsto n\min\{x,1/n\}$ and an approximately continuous function $g$. Furthermore, $h_n^{-1}(0)$ is a null cover of $N$ and there exists $n$ such that the measure of $\R\setminus h_n^{-1}(1)$ is less than 1.

\subsection{Jones functions}\label{Jones_Secion}
A function $f\colon \R\to \R$ is called a \textit{Jones function} if for every closed set $K\subseteq \R^2$ with an uncountable projection on the $x$-axis we have $f\cap K\neq \emptyset$. 
Equivalently, if for any perfect subset $P\subseteq \R$ and continuous function $g\colon P\to \R$ we have $f\cap g\neq \emptyset$.

Jones functions were introduced by F. B. Jones in \cite{Jones}. The author considered solutions of the Cauchy equation $f(x+y)=f(x)+f(y)$. He constructed a discontinuous function with a connected graph which satisfies the equation and the above definition.

One could ask if there is a function $F\colon\R^2\to \R$ for which all Carath\'{e}odory superpositions with continuous functions are Jones. Note that this problem is trivial -- it suffices to define $F(x, y)=g(x)$, where $g$ is any Jones function. Therefore we also consider inverted Carath\'{e}odory superpositions: we say that a function $F\colon\R^2\to \R$ is \textit{sup-Jones} if for every continuous function $f\colon\R\to\R$ the functions $F_f(x)\coloneqq F(x, f(x))$, $F^f(x)\coloneqq F(f(x), x)$ are Jones.

\subsection{Condition (A)}\label{SectionConitionA}
Consider the following condition.
\begin{enumerate}
    \item[(A)] There exists a function $f\colon \R\to \R$ which is a union of a family of pairwise disjoint partial functions $\{f_\alpha\colon \alpha<\C\}$ such that each $f_\alpha$ has positive outer measure and $\{x\in \R\colon f(x)=g(x)\}$ is a null set for each continuous $g\colon \R\to \R$.
\end{enumerate}

In \cite{Weizsacker} von Weizsäcker noted that if $\operatorname{non}(\mathcal{N})\coloneqq \min\{\vert A\vert\colon A\subseteq \R \text{ is not a null set}\}=\C$, then there exists a function $f\colon\R\to\R$ which has full outer measure and $\{x\in \R\colon f(x)=g(x)\}$ is a null set for every continuous $g\colon \R\to \R$. Under the same assumption, by \cite[Lemma 8]{Nat}, such a function can be decomposed into $\C$ many partial functions of full outer measure. In fact, with the assumption $\operatorname{non}(\mathcal{N})=\C$ such a family of pairwise disjoint partial functions can be defined in a similar way to von Weizsäcker's definition of a single function.

Later we will prove that condition (A) implies strong $2^\C$-algebraility of the family of all sup-measurable functions which are non-measurable. Thus, by the result of Rosłanowski and Shelah, \cite{RS} condition (A) is independent of ZFC. It is unclear to us whether condition (A) is equivalent to the existence of a non-measurable sup-measurable function or not.

\subsection{Condition $\operatorname{cov}(\mathcal{N})=\operatorname{add}(\mathcal{N})$}\label{Cov_is_Add_section}  The minimal cardinal number $\kappa$ such that the real line can be covered by $\kappa$ many null sets is denoted by $\operatorname{cov}(\mathcal{N})$ and it is between $\omega_1$ and $\mathfrak{c}$. Similarly, the minimal cardinal number $\kappa$ such that some union of $\kappa$ many null sets is not null is denoted by $\operatorname{add}(\mathcal{N})$. Clearly $\omega_1\leq\operatorname{add}(\mathcal{N})\leq \operatorname{cov}(\mathcal{N})$. The equality $\operatorname{cov}(\mathcal{N})=\operatorname{add}(\mathcal{N})$ means that $\R$ can be covered by $\kappa$ many null sets but any union of less than $\kappa$ many of them is null, where $\kappa$ is the common cardinal $\operatorname{cov}(\mathcal{N})$ and $\operatorname{add}(\mathcal{N})$. This condition is independent of ZFC, see \cite{BJ} for details. For example, it is fulfilled under CH, where both $\operatorname{cov}(\mathcal{N})$ and $\operatorname{add}(\mathcal{N})$ are $\omega_1$.

\subsection{Almost perfectly everywhere surjective functions and Bernstein sets}\label{SectionAPESnowa}

We say that $f\colon\mathbb{ R}\to\mathbb{R}$ is \textit{almost perfectly everywhere surjective} if its range  $f[\mathbb{R}]$ is one of the following: $\mathbb{R},[0,\infty)$, or $(-\infty,0]$, and $f[P]=f[\mathbb{R}]$ for any perfect set $P\subset\mathbb{R}$. Note that this notion is different from that of a perfectly everywhere surjective function known in the literature, cf. \cite{GMS}. Let us denote the family of all almost perfectly everywhere surjective functions by $\mathcal{APES}$. It follows from \cite[Theorem 2.2]{BGP} that $\mathcal{APES}$ is strongly $2^\mathfrak{c}$-algebrable.

This notion is connected to the following. A set $B\subseteq\R$ is called a \textit{Bernstein set} if $B\cap P\neq\emptyset$ and $(\R\setminus B)\cap P\neq\emptyset$ for every perfect subset $P$ of $\R$. It is known that such sets are non-measurable. Let $f$ be an almost perfectly everywhere surjective function. Then $f^{-1}(x)$ is a Bernstein set for every $x\in f[\mathbb{R}]$. Therefore $f$ is non-measurable and $\{f^{-1}(x)\colon x\in f[\R]\}$ is a partition of $\R$.

\subsection{Approximately differentiable and nowhere approximately differentiable functions}\label{SectionApproxDiff} A function $f\colon \mathbb{R} \rightarrow \mathbb{R}$ is said to be \textit{approximately differentiable} at a point $x \in \mathbb{R}$ if there exists a measurable set $E \subset \mathbb{R}$ such that $x$ is its density point, and the restriction $f \upharpoonright E$ is differentiable at $x$. A function $f\colon \mathbb{R} \rightarrow \mathbb{R}$ is \textit{nowhere approximately differentiable} if it is approximately differentiable at no $x \in \mathbb{R}$.

We will use the following observation. Let $g$ be a continuous and almost everywhere differentiable function, and let $f$ be a continuous and nowhere approximately differentiable function. Then the set $E\coloneqq\{x\in\mathbb{R}\colon f(x)=g(x)\}$ has measure zero. To see this, first note that $E$ is closed, since both $g$ and $f$ are continuous. Suppose, on the contrary, that $E$ has positive measure. By the Lebesgue Density Theorem, $E$ has a density point, say $x$, at which $g$ is differentiable. This implies that $f$ is approximately differentiable at $x$, which is a contradiction. 

\subsection{Grande's construction of non-measurable function with Darboux Baire one sections.}\label{sekcjaGrande}
In \cite{Lipinski} Lipi\'{n}ski constructed an example of a non-measurable function $F\colon\R^2\to \R$ whose all vertical and horizontal sections are Darboux Baire one functions. Another such construction is due to Grande \cite[Theorem 2]{Grande}. The paper is written in French and is therefore not easily accessible. Here we present Grande's construction, slightly modified for our purposes. The difference is this. The function $h\colon X \to \left(0,1\right]$ used below was constant and equal to $1$ in Grande's original construction.

Let $C\subseteq [0 ,1]$ be a Cantor set of positive measure with $0, 1\in C$. Let $a_0=0$, $b_0=1$ and $\{(a_n, b_n)\colon n\geq 1\}$ be an enumeration of the gaps.  Define
\[g(x)=
\begin{cases*}
      g_n(x) & if $x\in (a_n, b_n)$ for some $n\geq 1$, \\
      0 & otherwise,
    \end{cases*}
    \]
where $g_n\colon (a_n, b_n)\to (0, 1]$, $n\geq 1$ are continuous surjections onto $(0, 1]$, with
\[
\lim_{x\rightarrow a_n^+} g_n(x)= \lim_{x\rightarrow b_n^-} g_n(x) = 0.
\]
Due to the density of the gaps in $C$, any modification of the function $g$ on $C$ with values in $[0, 1]$ preserves the intermediate value property.
Let $B\subseteq C\setminus \bigcup_{n\geq 0} \{a_n, b_n\}$ 
be a closed set with positive measure.

Let $X\subseteq B\times B$ be a non-measurable set in which all vertical and horizontal sections have at most one element (e.g. $X=\bigcup_{\alpha<\C}A_\alpha$, where $A_\alpha, \alpha<\C$, are like in Lemma \ref{rodzinaAalfa} for $Y=B\times B$). Let $h\colon X\to (0, 1]$ be any function. Define
\[
F_h(x, y) =
    \begin{cases*}
      g(x) & if $x\in \R\setminus B$, \\
      g(y) & if $x\in B$ and $(x, y)\notin X$, \\
      h(x, y) & if $(x, y)\in X$.
    \end{cases*}
\]

Note that $F_h^{-1}(0)\cap(C\times C)=(C\times C)\setminus X$, so $F_h$ is non-measurable.

We will show that all vertical and horizontal sections of $F$ are Darboux Baire one functions.


Let $x\in \R$. If $x\in \R\setminus B$, then $F_h(x, \cdot)$ is constant. Assume that $x\in B$. Then $F_h(x, \cdot)$ is either $g$ or $g$ modified at some point $y\in C$, where $(x, y)\in X$, so it has the intermediate value property. By Lemma \ref{lemacikBaireOne}, $F_h(x, \cdot)$ is also Baire one. 

Let $y\in \R$. 
Consider the function 
\[
\ell(x) =
    \begin{cases*}
      g(x) & if $x\in \R\setminus B$, \\
      g(y) & if $x\in B$.
    \end{cases*}
\]
Note that $F_h(\cdot, y)$ is either $\ell$ or $\ell$ modified at one point $x\in C$, provided that $(x, y)\in X$, so it has the intermediate value property. By Lemma \ref{lemacikBaireOne}, $F_h(\cdot, y)$ is also Baire one.

\subsection{Exponential like functions}\label{exponentialLIKE}
We say that a function $f\colon\R\to \R$ is \textit{exponential like} if 
\[
f(x)=\sum_{i=1}^n \alpha_i e^{\beta_i x}, x\in \R
\]
for some positive integer $n$, non-zero real numbers $\alpha_1, \dots, \alpha_n$ and distinct non-zero real numbers $\beta_1, \dots, \beta_n$. The notion was described in \cite{Balcerzak}, where the Authors proved that if $\mathcal{A}\subseteq \R^\R$ is an arbitrary  family and there exists a function $F\in \mathcal{A}$ such that $f\circ F\in \mathcal{A}$ for every exponential like function $f$, then $\mathcal{A}$ is strongly $\C$-algebrable.

\subsection{Ultrafilters on $\w$}
By an \textit{utrafilter} on $\w$ we mean any maximal non-trivial family of subsets of $\w$ which is closed under taking supersets and finite intersections. Endowing $\w$ with the discrete topology, we denote by $\beta \w$ its Stone-Čech compactification, that is, the set of all ultrafilters on $\w$ endowed with the topology which basic sets are of the form
\begin{equation*}
\beta a=\{ U \in \beta \w \colon a \in U\},
\end{equation*}
where $a \subset \w$. We will identify $\w$ with the family of all principal ultrafilters $\delta_n=\{a \subset \w \colon n \in a\}$, $n\in \w$.

Using the fact that if $X$ is a compact space and
$f \colon \w \to X$ is a continuous function, then
there exists a continuous extension $\overline{f}\colon \beta \w \to X$ of $f$ (see e.g. \cite{engelking}), we prove the following.
 
\begin{lemma}\label{lematOrozszerzaniu}
 Let $f \colon \w \to m$ be any function. Let
$\overline{f}$ be a continuous extension of $f$, let $i<m$ and $u=f^{-1}(i)$. Then $\overline{f}(U)=i$ for every $U\in \beta u$.   
\end{lemma}

\begin{proof}
    The extenstion $\overline{f}$ exists as $f$ is continuous. Take any $U \in \beta u$ and suppose that $\overline{f}(U) \neq i$.
Take a neighbourhood $V$ of $\overline{f}(U)$ with $i\notin V$. Then there is basic open set $\beta b$ with $U\in \beta b \cap \beta u \subset \overline{f}^{-1}[V]$. By the density of $\w$ in $\beta\w$, there is $k\in \w \cap \beta u \cap \beta b \subset \overline{f}^{-1}[V]$, so $\overline{f}(k)\neq i$. However, $f(k)=i$ because $k \in u$, a contradiction.

\end{proof}


\section{Results}\label{Res}
\begin{theorem}\label{Mateusz}
The family of all measurable functions that are not sup-measurable is strongly $2^\C$-algebrable.
\end{theorem}

\begin{proof} 
$C$ is the Cantor ternary set. By $\{h_\xi\colon\xi<2^\mathfrak{c}\}$ we denote a set of free generators of an algebra in $\R^C$. Let $X\subseteq \R$ be a non-mesurable set. For $\xi<\mathfrak{c}$ we define
\[
F_\xi(x, y) = h_\xi(y)\chi_{X\times C}(x, y).
\]
Then $\{(x, y)\in \R^2\colon F_\xi(x, y)\neq 0\}$ is a null set for each $\xi <\C$, so functions $F_\xi$ are measurable.

For $\xi_1<\xi_2<\dots<\xi_k$ and polynomial $P$ in $k$ variables without constant term we have 
\[
F(x, y)\coloneqq P(F_{\xi_1},F_{\xi_2},\dots,F_{\xi_k})(x, y)=P(h_{\xi_1}(y),h_{\xi_2}(y),\dots,h_{\xi_k}(y))\chi_{X\times C}(x, y).
\]
Let $y_0\in C$ be such that $P(h_{\xi_1}(y_0),h_{\xi_2}(y_0),\dots,h_{\xi_k}(y_0))$ is non-zero. Define $g$ to be a constant function, for $x\in \R$
\[
g(x)= y_0.
\]
Then 
\[
F_g=P(h_{\xi_1}(y_0),h_{\xi_2}(y_0),\dots,h_{\xi_k}(y_0))\chi_X
\]
is a scaled characteristic function of a non-measurable set $X$, so $F$ is not sup-measurable. In particular, $F$ is non-zero.
\end{proof}

\begin{theorem}\label{SupMeasurableNonMeasurableNewAssumption}
Assume (A). Then the family of all non-measurable sup-measurable functions is strongly $2^\mathfrak{c}$-algebrable. 
\end{theorem}

\begin{proof} Let $f$ and $f_\alpha$, $\alpha<\mathfrak c$, be as in (A).  
By $\{h_\xi\colon\xi<2^\mathfrak{c}\}$ we denote a set of free generators of an algebra in $\R^\mathfrak{c}$. For $\xi<\mathfrak{c}$ we define
\[
F_\xi(x, y)= \sum_{\alpha<\mathfrak{c}} h_\xi(\alpha)\chi_{f_\alpha}(x, y).
\]
For $\xi_1<\xi_2<\dots<\xi_k$ and polynomial $P$ in $k$ variables without constant term we have 
\[
F\coloneqq P(F_{\xi_1},F_{\xi_2},\dots,F_{\xi_k})=\sum_{\alpha<\mathfrak{c}}P(h_{\xi_1}(\alpha),h_{\xi_2}(\alpha),\dots,h_{\xi_k}(\alpha))\chi_{f_\alpha}
\]
Since $P(h_{\xi_1}(\beta),h_{\xi_2}(\beta),\dots,h_{\xi_k}(\beta))$ is non-zero for some $\beta<\mathfrak{c}$, we have
\[
f_\beta\subseteq F^{-1}(P(h_{\xi_1}(\beta),h_{\xi_2}(\beta),\dots,h_{\xi_k}(\beta)))\subseteq f.
\]
Therefore $F^{-1}(P(h_{\xi_1}(\beta),h_{\xi_2}(\beta),\dots,h_{\xi_k}(\beta)))$ has positive outer measure and null vertical sections (as a subset of a graph of a function), so, by Fubini's Theorem, it is non-measureable.
In particular, $F$ is non-zero.

Let $g\colon\R\to\R$ be a continuous function. Note that if $F(y,g(y))\neq 0$, then $y\in\{x\in \R\colon f(x)=g(x)\}$, which is a null set. So $x\mapsto F(x,g(x))$ is measurable, and consequently $F$ is sup-measurable.  
\end{proof}
\begin{theorem}\label{WeaklySupmeasurableNonmeasurableNonSupmeasurable}
Assume (A). Then the family of all weakly sup-measurable functions that are neither
sup-measurable nor measurable is strongly $2^\mathfrak{c}$-algebrable.
\end{theorem}

\begin{proof} Let $f$ and $f_\alpha$, $\alpha<\mathfrak c$, be as in (A).
Let $\{h_\xi\colon\xi<2^\mathfrak{c}\}$ be a set of free generators of an algebra in $\R^\mathfrak{c}$. 
Let $\{p_{\xi}\colon\xi<2^\mathfrak{c}\}$ be a set of free generators spanning an algebra in $\mathcal{APES}\cup\{0\}$ (in fact, we could replace $\mathcal{APES}$ by any strongly $2^\mathfrak{c}$-algebrable family of non-measurable functions).
Let $h\colon\R\to\R$ be a continuous nowhere approximately differentiable function.  
For $\xi<2^\mathfrak{c}$ we define $G_\xi\colon\R^2\to\R$ as follows
\[
G_\xi(x,y)=\sum_{\alpha<\mathfrak{c}}h_\xi(\alpha)\chi_{f_\alpha\setminus h}(x,y)+p_\xi(x)\chi_{h}(x,y).
\]
For $\xi_1<\xi_2<\dots<\xi_k$ and polynomial $P$ in $k$ variables without constant term we have 
\[
G\coloneqq P(G_{\xi_1},\dots,G_{\xi_k})=\sum_{\alpha<\mathfrak{c}}P(h_{\xi_1}(\alpha),\dots,h_{\xi_k}(\alpha))\chi_{f_\alpha\setminus h}+P(p_{\xi_1}(x),\dots,p_{\xi_k}(x))\chi_{h}(x,y).
\]
We need to show that $G$ is weakly sup-measurable, non-measurable and is not sup-measurable (then clearly $G$ is also non-zero).

We already know that $F\coloneqq \sum_{\alpha<\mathfrak{c}}P(h_{\xi_1}(\alpha),h_{\xi_2}(\alpha),\dots,h_{\xi_k}(\alpha))\chi_{f_\alpha}$ is non-measurable -- see the proof of Theorem \ref{SupMeasurableNonMeasurableNewAssumption}. Note that $\{(x,y)\in\mathbb{R}^2\colon F(x,y)\neq G(x,y)\}\subseteq\{(x,h(x))\colon x\in\mathbb{R}\}$. Since the graph of $h$ has measure zero, then $G$ is also non-measurable. 

Let us show that $G$ is not sup-measurable. Consider $G_h(x)=G(x,h(x))=p(x)$ where $p(x)\coloneqq P(p_{\xi_1}(x),\dots,p_{\xi_k}(x))$ is almost perfectly everywhere surjective. As we have noticed in Section \ref{SectionAPESnowa}, almost perfectly everywhere surjective functions are non-measurable. Therefore $G$ is not sup-measurable. 

To finish the proof we need to check that $G$ is weakly sup-measurable. To do this, we fix a continuous almost everywhere differentiable function $g$. Then 
\[
G_g(x) =
    \begin{cases*}
      F_g(x) & if $h(x)\neq g(x)$, \\
      p(x) & if $h(x)=g(x)$. 
    \end{cases*}
\]
Consider the set $E\coloneqq \{x\in\mathbb{R}\colon g(x)=h(x)\}$. As we have noticed in Section \ref{SectionApproxDiff}, $E$ is a null set. This shows that $G_g$ and $F_g$ are equal on a set of full measure. We have shown in the proof of Theorem \ref{SupMeasurableNonMeasurableNewAssumption} that $F_g$ is measurable, and so is $G_g$. Therefore $G$ is weakly sup-measurable.
\end{proof}

\begin{corollary}\label{WeaklySupMEasurable_non_sub_measurable}
Assume (A). The family of all weakly sup-measurable functions that are not sup-measurable is strongly $2^\mathfrak{c}$-algebrable.
\end{corollary}

\begin{theorem}\label{Non_measurable_sep_measurable}
The family of all non-measurable separately measurable functions is strongly $2^\mathfrak{c}$-algebrable.
\end{theorem}

\begin{proof}
Let $\{h_\xi\colon\xi<2^\mathfrak{c}\}$ denote a set of free generators of an algebra in $\R^\mathfrak{c}$. Let $\{A_\alpha\colon\alpha<\mathfrak{c}\}$ be a family described in 
Lemma \ref{rodzinaAalfa} (for $Y=\R^2$).
Let $A=\bigcup_{\alpha<\mathfrak{c}}A_\alpha$. For $\xi<2^\mathfrak{c}$ we define $F_\xi\colon\R^2\to\R$ as follows
\[
F_\xi(x,y)=\sum_{\alpha<\mathfrak{c}}h_\xi(\alpha)\chi_{A_\alpha}(x,y).
\]
For $\xi_1<\xi_2<\dots<\xi_k$ and polynomial $P$ in $k$ variables without constant term we have 
\[
F\coloneqq P(F_{\xi_1},F_{\xi_2},\dots,F_{\xi_k})=\sum_{\alpha<\mathfrak{c}}P(h_{\xi_1}(\alpha),h_{\xi_2}(\alpha),\dots,h_{\xi_k}(\alpha))\chi_{A_\alpha}.
\]
Let us show the $F$ is non-measurable. There exists $\beta<\mathfrak{c}$ such that  $P(h_{\xi_1}(\beta),h_{\xi_2}(\beta),\dots,h_{\xi_k}(\beta))\neq 0$. Then
\[
A_\beta\subseteq F^{-1}(P(h_{\xi_1}(\beta),h_{\xi_2}(\beta),\dots,h_{\xi_k}(\beta)))\subseteq A.
\]


Therefore $F^{-1}(P(h_{\xi_1}(\beta),\dots,h_{\xi_k}(\beta)))$ has positive outer measure (as a superset of $A_\beta$) and null vertical sections (as a subset of $A$), so, by Fubini’s Theorem, it is non-measureable. Consequently, $F$ is non-measurable. In particular, $F$ is non-zero.

Since each vertical and horizontal section of $A$ has at most two elements and $\{(x,y)\colon F(x,y)\neq 0\}\subseteq A$, then, using what we observed in Section \ref{SectionSeparatelyMeasurable}, we get that $F$ is separately measurable. 
\end{proof}

\begin{theorem}\label{TheoremDarbouxSections}
The family of all non-measurable functions $F\colon\R^2\to\R$ whose all vertical and horizontal sections are Darboux Baire one is strongly $\C$-algebrable.
\end{theorem}

\begin{proof}
Here we follow the notation from Section \ref{sekcjaGrande}. Let $\{A_\alpha\colon \alpha<\C\}$ be a family described in Lemma \ref{rodzinaAalfa} for $Y=B\times B$ Let $h\colon X\to(0,1]$ be defined as follows
\[
h(x, y)=\sum_{\alpha<\C} r_\alpha\chi_{A_\alpha},
\]
where $\{r_\alpha\colon \alpha<\C\}$ is a one-to-one enumeration of $(0, 1]$.
We will show that the composition $f\circ F_h$ with any exponential like function $f$ is a non-measurable function with Darboux Baire one sections, which implies strong $\C$-algebrability of the considered family (see Section \ref{exponentialLIKE}). 

Indeed, let $y\in \R$ and $f$ be any exponential like function. Note that 
\[
(f\circ F_h)(\cdot, y)=f\circ (F_h(\cdot, y)).
\]
Therefore $(f\circ F_h)(\cdot, y)$ is a Darboux Baire one as a composition of $F_h(\cdot, y)$ with a continuous function. Similarly for vertical sections. 

Now choose $\beta<\C$ such that $f(r_\beta)\neq f(0)$. This can be done because an exponential like function is not constant on any open interval, by the identity theorem for analytic functions. Then $A_\beta\subseteq (f\circ F_h)^{-1}(f(r_\beta))\cap B\times B \subseteq \bigcup_{\alpha<\C} A_\alpha$, so $(f\circ F_h)^{-1}(f(r_\beta))\cap B\times B$ has full outer measure (in $B\times B)$ and 
null sections. According to Fubini's Theorem, this set is non-measurable. So $f\circ F_h$ is non-measurable.
\end{proof}


\begin{theorem}\label{Non_Measurable_all_vert_sec_app_cont_and_all_hor_sec_meas}
Assume $\operatorname{cov}(\mathcal{N})=\operatorname{add}(\mathcal{N})$. Then the family of all non-measurable functions having all vertical sections approximately continuous and all horizontal sections measurable, is strongly $2^\mathfrak{c}$-algebrable. 
\end{theorem}

\begin{proof}
Let $\kappa=\operatorname{cov}(\mathcal{N})=\operatorname{add}(\mathcal{N})$, and let $\mathbb{R}=\bigcup_{\alpha<\kappa} C_\alpha$, where $C_\alpha$ are null sets. For every $\alpha<\kappa$, the set $D_\alpha\coloneqq\bigcup_{\beta \leq \alpha} C_\beta$ has measure zero. There is an approximately continuous $g_\alpha\colon \mathbb{R} \rightarrow[0,1]$ such that $g_\alpha^{-1}(0)$ is a null cover of $D_\alpha$ and $\R\setminus g_\alpha^{-1}(1)$ has measure less than 1 (see Section \ref{Appr_Cont_Section}). By $\{h_\xi\colon\xi<2^\mathfrak{c}\}$ we denote a set of free generators of an algebra in 
$\R^\C$.
Let $\{B_\alpha\colon\alpha<\mathfrak{c}\}$ be a family of pairwise disjoint Bernstein sets (see Section \ref{SectionAPESnowa}). For each $r \in \mathbb{R}$ let $\alpha(r)$ denote the first ordinal $\alpha$ with $r \in C_\alpha$. We define
\[
F_\xi(x, y)= g_{\alpha(x)}(y)\sum_{\beta<\mathfrak{c}}h_\xi(\beta)\chi_{B_\beta}(x) .
\]
%
%
Fix $x\in \R$. If $x\notin \bigcup_{\beta<\C} B_\beta$, then $F_\xi(x, \cdot)$ is approximately continuous as a constant zero function. If $x\in B_\beta$ for some $\beta<\C$, then $F_\xi(x, \cdot)=h_\xi(\beta)g_{\alpha(x)}$, so $F_\xi(x, \cdot)$ is approximately continuous.
Fix $y\in\R$ and assume that $x\notin\bigcup_{\beta<\alpha(y)}C_\beta$. Then $\alpha(x)\geq\alpha(y)$ and $g_{\alpha(x)}$ vanishes at 
\[
D_{\alpha(x)}=\bigcup_{\beta\leq\alpha(x)}C_\beta\supseteq \bigcup_{\beta\leq\alpha(y)}C_\beta\supseteq C_{\alpha(y)}\ni y.
\]
So $F(\cdot, y)=0$ almost everywhere. 
For $\xi_1<\xi_2<\dots<\xi_k$ and polynomial $P$ in $k$ variables without constant term, let $F= P(F_{\xi_1},F_{\xi_2},\dots,F_{\xi_k}).$ Since the sum and the product of two approximately continuous functions is approximately continuous, then, by simple induction, $F(x, \cdot)$ is approximately continuous. 

Since $F(\cdot, y)=0$ almost everywhere for every $y\in\R$, then 
\[
\int_{\mathbb{R}} \left(\int_{\mathbb{R}} F(x, y) \mathrm{d} x\right) \mathrm{d} y=0.
\]
There exists $\beta<\mathfrak{c}$ such that $P(h_{\xi_1}(\beta),h_{\xi_2}(\beta),\dots,h_{\xi_k}(\beta))\neq 0$. For each $x\in B_\beta$ we have
\[
\int_{\mathbb{R}} F(x, y) \mathrm{d} y=\int_{\mathbb{R}\setminus (g_{\alpha(x)})^{-1}(1)} F(x, y) \mathrm{d} y+\int_{(g_{\alpha(x)})^{-1}(1)} P(h_{\xi_1}(\beta),h_{\xi_2}(\beta),\dots,h_{\xi_k}(\beta)) \mathrm{d} y.
\]

Note that the absolute value of the first integral is not greater than 
\[\max\{\vert P(h_{\xi_1}(\beta)g_{\alpha(x)}(y),h_{\xi_2}(\beta)g_{\alpha(x)}(y),\dots,h_{\xi_k}(\beta)g_{\alpha(x)}(y))\vert\colon 0\leq g_{\alpha(x)}(y)\leq 1\}
\]
\[
\leq\max\{\vert P(h_{\xi_1}(\beta)t,h_{\xi_2}(\beta)t,\dots,h_{\xi_k}(\beta)t)\vert\colon 0\leq t\leq 1\}
\]
while the second integral is infinite and has  the same sign as $P(h_{\xi_1}(\beta),h_{\xi_2}(\beta),\dots,h_{\xi_k}(\beta))$. So $\int_{\mathbb{R}} F(x, y) \mathrm{d} y$ is infinite for $x$'s from the Bernstein set $B_\beta$. Therefore, the iterated integral $\int_{\mathbb{R}} \left(\int_{\mathbb{R}} F(x, y) \mathrm{d} y\right) \mathrm{d} x$ is not zero, and, according to Fubini's Theorem, $F$ is non-measurable.

\end{proof}

\begin{theorem}\label{Jonesy}
    The family of sup-Jones functions is $2^\C$-lineable.
\end{theorem}
\begin{proof}
Let $\mathcal{L}^n$ be the family of all linear functionals defined on $\R^n$. 
Let 
\[
\mathcal{L}= \bigcup_{n\geq 1} \mathcal{L}^n \times n^{\w}.
\]
Note that the cardinality of $\Li$ is continuum. Let $\mathcal{K}$ be a family of all partial real continuous functions with perfect domain and $\mathcal{F}$ be a family of all continuous functions.
Note that the cardinality of $\mathcal{K} \times \mathcal{F} \times \Li$ is continuum.
Let $\mathcal{K} \times \mathcal{F} \times \Li =\{(g_{\alpha}, f_{\alpha}, l_{\alpha}, p_{\alpha})\colon\alpha < \C\}.$
Formally we should write $(g_{\alpha}, f_{\alpha}, (l_{\alpha}, p_{\alpha}))$ but we omit the inner parentheses for clarity.
For each $\alpha < \C$, let $K_{\alpha}$ be the domain of $g_{\alpha}$ and let $x_{\alpha} \in K_\alpha \setminus \left(\{f_{\xi}(x_{\xi})\colon \xi<\alpha\} \cup \{x_{\xi}\colon\xi< \alpha\}\right)$.
For an element $l_\alpha$ in $\Li^n$, we find $\overrightarrow{x_{\alpha}} \in \R^n$ such that $l_{\alpha}(\overrightarrow{x_{\alpha}})=g_{\alpha}(x_{\alpha})$.
Note that $p_{\alpha} \in n^\w$ is continuous (as a mapping between two discrete spaces), so we can consider its continuous extension $\overline{p_{\alpha}}\colon\beta \w \to n$.


For $U \in \beta \w$ we define a function $F_U\colon\R^2 \to \R$ in the following way:
\[
F_U(x_{\alpha}, f_{\alpha}(x_{\alpha}))=F_U(f_{\alpha}(x_{\alpha}), x_{\alpha})=\overrightarrow{x_{\alpha}} \circ \overline{p_{\alpha}}(U)
\]
for $\alpha < \C$, and $F_U$ takes $0$ at other points.



Let $n\geq 1$ and take a continuous $f\colon \R \to \R$, $g \in \mathcal{K}$, $l\in \Li^n$ and distinct $U_0, U_1, \dots, U_{n-1} \in \beta \w$.
We can find a partition $\{u_0, u_1, \dots, u_{n-1}\}$  of $\w$ such that $u_i \in U_i$ for $i=0, 1, \dots, n-1$. We define a function $p\colon\w \to n$ by the formula $p(k)=i \iff k \in u_i$ for $i=0, 1, \dots, n-1$. Take $\alpha < \C$ such that $g_\alpha=g, f_{\alpha}=f, l_{\alpha}=l, p_{\alpha}=p$. Then $\overline{p_{\alpha}}(U_i)=i$ for $i=0, 1, \dots, n-1$ (Lemma \ref{lematOrozszerzaniu}). Therefore
\[l(F_{U_0}, F_{U_1}, \dots, F_{U_{n-1}})(x_{\alpha}, f_{\alpha}(x_{\alpha}))=l(F_{U_0}, F_{U_1}, \dots, F_{U_{n-1}})(f_{\alpha}(x_{\alpha}), x_{\alpha})
\]
\[=
l_{\alpha}(F_{U_0}(x_{\alpha}, f_{\alpha}(x_{\alpha})), \dots, F_{U_{n-1}}(x_{\alpha}, f_{\alpha}(x_{\alpha})))=l_{\alpha}(\overrightarrow{x_{\alpha}} \circ \overline{p_{\alpha}}(U_0), \dots, \overrightarrow{x_{\alpha}} \circ \overline{p_{\alpha}}(U_{n-1}))
\]
\[=
l_{\alpha}(\overrightarrow{x_{\alpha}}(0), \dots, \overrightarrow{x_{\alpha}}(n-1))=l_{\alpha}(\overrightarrow{x_{\alpha}}) = g_{\alpha}(x_{\alpha}).
\]
So $l(F_{U_0}, \dots, F_{U_{n-1}})$ is sup-Jones. This proves that sup-Jones functions are $2^{\C}$-lineable (as $\vert \beta \w \vert = 2^\C$, see e.g. \cite{engelking}).
\end{proof}
    
Note that if $F$ is sup-Jones, then $F^2$ is not. Therefore the family of all sup-Jones functions is not $1$-algebrable.
We can modify the definition of sup-Jones functions to obtain strong $2^\C$-algebrability.
We say that $F\colon\R^2 \to \R$ is \textit{symmetric sup-Jones} function if for every continuous function $f\colon\R \to \R$ and every continuous real valued function $g$ defined on a perfect subset of $\R$, there exist 
$x, y$ in the domain of $g$
such that
$F(x, f(x)) \in \{g(x), -g(x)\}$ and $F(f(y), y) \in \{g(y), -g(y)\}$.
By replacing linear mappings by polynomials without constant terms and sup-Jones functions by symmetric sup-Jones functions in the proof of Theorem \ref{Jonesy}, we obtain the proof of the strong $2^\C$-algebrability of the family of all symmetric sup-Jones functions.

\bibliographystyle{abbrv}
\bibliography{bibliography}

\end{document}